\documentclass[11pt]{article}
\usepackage[utf8]{inputenc}
\usepackage{amsfonts}
\usepackage{amsmath} 
\usepackage{amssymb}
\usepackage{amsthm}
\usepackage{tikz}
\usepackage[margin=.95in]{geometry}
\usepackage{mathtools}

\newtheorem{theorem}{Theorem}
\newtheorem{lemma}[theorem]{Lemma}
\newtheorem{proposition}[theorem]{Proposition}
\newtheorem{corollary}[theorem]{Corollary}

\newtheorem{definition}[theorem]{Definition}

\newtheorem{question}[theorem]{Question}

\newcommand{\bth}{\begin{theorem}}
\renewcommand{\eth}{\end{theorem}}
\newcommand{\bpr}{\begin{proposition}}
\newcommand{\epr}{\end{proposition}}
\newcommand{\bde}{\begin{definition}}
\newcommand{\ede}{\end{definition}}
\newcommand{\blem}{\begin{lemma}}
\newcommand{\elem}{\end{lemma}}
\newcommand{\bco}{\begin{corollary}}
\newcommand{\eco}{\end{corollary}}
\newcommand{\prove}{\begin{proof}}
\newcommand{\done}{\end{proof}}

\newcommand{\ite}{\begin{itemize}}
\newcommand{\mize}{\end{itemize}}

\newcommand{\ben}{\begin{enumerate}}
\newcommand{\een}{\end{enumerate}}

\DeclareMathOperator{\st}{st}

\title{Pattern Matching in Set Partitions is NP-Complete}
\author{Thomas Grubb\\\texttt{tgrubb@ucsd.edu}\\University of California, San Diego}
\date{\today\\[10pt]
	\begin{flushleft}
	\small Key Words: Avoidance, complexity, pattern matching, permutations, set partitions 
	                                       \\[5pt]
	\small AMS subject classification (2010): 03D15, 05A18
	\end{flushleft}}
\begin{document}
\maketitle
\abstract{In this note we show that pattern matching in permutations is polynomial time reducible to pattern matching in set partitions. In particular, pattern matching in set partitions is NP-Complete.}
\vspace{10pt}
\section{Introduction}
A permutation $\pi = \pi_1\dots\pi_n$ \emph{contains} a permutation $\tau = \tau_1\dots \tau_k$ if  there exist indices $i_1<\dots< i_k$ for which $\pi_{i_1}\dots \pi_{i_k}$ is \emph{order isomorphic} to $\tau$, i.e.
$$
\pi_{i_j}<\pi_{i_{j'}}\text{ if and only if }\tau_{j}<\tau_{j'}.
$$
If $\pi$ does not contain $\tau$ we say $\pi$ \emph{avoids} $\tau$. Determining whether or not an arbitrary permutation $\pi$ contains an arbitrary permutation $\tau$ is known as the \emph{permutation pattern matching problem} (PPM).

The complexity of pattern matching in permutations has been extensively studied. Bose, Buss, and Lebiw initialized this study by showing that permutation pattern matching on generic inputs is NP-complete; moreover, counting the number of occurences of $\tau$ inside $\pi$ is $\#$P-complete \cite{BBL}. They achieve this by reducing the Boolean 3-satisfiability problem to PPM. Subsequent developments have been made by placing restrictions on one or both of $\pi$ or $\tau$. Such results include both algorithm development and theoretical hardness results; see, for instance, \cite{AAAH}, \cite{ALLV}, \cite{AR}, \cite{BL}, \cite {GM}, \cite{GV}, \cite{Iba}, \cite{JK}, and \cite{YS}. These results can be pragmatic in that they may allow for efficient collection of data regarding permutation patterns; such data can be used to formulate conjectures in the field of permutation patterns. See, for instance, the database of Tenner \cite{DPPA}. For a more complete introduction to permutation patterns we recommend \cite{Kit}.

To date the hardness of pattern matching has not received much interest in other combinatorial contexts. We hope that this note may spur interest regarding pattern matching in \emph{set partitions}. Given a set $S$, a set partition of $S$ is an unordered collection of disjoint \emph{blocks} $\{B_1,\dots,B_k\}$ for which $B_1\cup\dots \cup B_k = S$. We will be concerned with set partitions of $[n]:=\{1,2,\dots,n\}$ for a positive integer $n$. 

Given a subset $T\subset [n]$, define the \emph{standardization map} $\st:T\to [\#T]$ as the map which sends the $i$th smallest element of $T$ to $i$. Given a partition $\sigma = \{B_1,\dots,B_k\}$ of $[n]$ and $T\subset [n]$ define the partition $\sigma\cap T$ of $[\#T]$ as the partition whose blocks are the nonempty sets of the form $\st(B_i \cap T)$. For example, if $\sigma = \{\{1,3\},\{2,4\}\}$ and $T = \{1,3,4\}$ then $$\sigma \cap T = \{\{1,2\},\{3\}\}.$$ 
Given a partition $\sigma$ of $[n]$ and a partition $\sigma'$ of $[k]$ we say $\sigma$ \emph{contains} $\sigma'$ if there is a subset $T\subset [n]$ for which $\sigma\cap T = \sigma'$. If this is not the case, then we say $\sigma$ \emph{avoids} $\sigma'$.

The combinatorics of set partition patterns has been well studied; see, for instance, \cite{Sag} and references therein. It is known that set partition patterns can also occasionally shed light to the theory of permutation patterns, specifically in the context of combinatorial statistics \cite{LF}. The main result of this note is to introduce another relation between set partition patterns and permutation patterns and to encourage the exploration of the complexity of algorithmic questions regarding set partition patterns. To this end we call the problem of determining whether or not a partition $\sigma$ contains a pattern $\sigma'$ the \emph{set partition pattern matching problem} (SPPM). 

The main result of this paper is as follows. The proof will come in Section 2.
\begin{theorem}
The PPM problem is polynomial time reducible to the SPPM problem. 
\end{theorem}
In particular, this implies that SPPM is NP-Complete and that the corresponding counting problem is $\#$P-Complete. More than this it shows that techniques for examining set partition patterns can potentially be ported over to give insight into permutation patterns. We hope that this may spark further interest and research into the complexity of pattern matching in a wider combinatorial context.
\section{Main Result}
The sole purpose of this section is to prove Theorem 1. We will follow this section with a concluding section in which we discuss several corollaries and possible avenues for future research. 
\begin{proof}[Proof of Theorem 1]
To prove this result we take as input two permutations $\pi = \pi_1\dots \pi_n$ and $\tau = \tau_1\dots \tau_k$ and produce two partitions $s(\pi)$ and $s(\tau)$ with the following properties:
\begin{itemize}
    \item $\pi$ contains $\tau$ if and only if $s(\pi)$ contains $s(\tau)$, and 
    \item $s(\pi)$ and $s(\tau)$ are partitions of $[2n]$ and $[2k]$ respectively.
\end{itemize}
The first property ensures that an algorithm which can match set partitions can match permutations. The second property ensures that in passing between the two problems our input size is only doubled. We claim the partitions 
\begin{align*}
    s(\pi) &= \{\{1,\pi_1+n\},\{2,\pi_2+n\},\dots,\{n,\pi_n+n\}\},\\
    s(\tau) &= \{\{1,\tau_1+k\},\{2,\tau_2+k\},\dots,\{k,\tau_k+k\}\}
\end{align*}
satisfy the desired properties. Clearly they are of the desired size. 

Assume that $\pi$ contains $\tau$, and let $\pi_{i_1}\dots \pi_{i_k}$ be an occurrence of $\tau$ in $\pi$. Let $T\subset [2n]$ be the set $$T=\{i_1,\dots,i_k,\pi_{i_1}+n,\dots, \pi_{i_k}+n\}.$$
The partition $T\cap \pi$ is given by
$$
T\cap s(\pi) = \{\st(\{i_1,\pi_{i_1}+n\}),\dots,\st(\{i_k,\pi_{i_k}+n\})\}.
$$
The fact that $T\cap s(\pi) = s(\tau)$ follows from the next three facts:
\begin{itemize}
    \item $i_1<\dots<i_k$,
    \item $\max(\{i_1, \dots, i_k\}) < n+\min(\{\pi_{i_1}, \dots ,\pi_{i_k}\})$,
    \item $\pi_{i_j}<\pi_{i_{j'}}$ if and only if $\tau_{j} < \tau_{j'}$.
\end{itemize}
These follow directly from the definitions.

For the reverse direction, we show that if $s(\pi)$ contains $s(\tau)$ then $\pi$ contains $\tau$. Accordingly let $T\subset [2n]$ be such that $T\cap s(\pi) = s(\tau)$. As both $s(\pi)$ and $s(\tau)$ consist solely of blocks of size $2$, $T\cap s(\pi)$ must be formed by standardizing a subset of the blocks of $s(\pi)$. In other words $T$ must be of the form $T= S_1\cup S_2$, with 
\begin{align*}
    S_1 &= \{i_1,\dots, i_k\} \subset [n],\\
    S_2 &= \{\pi_{i_j}+n: i_j \in S_1\}.
\end{align*}
We may assume without loss of generality that $i_1<\dots< i_k$; in this case it is straightforward to see that $\pi_{i_1}\dots \pi_{i_k}$ gives an occurrence of $\tau$ inside $\pi$ as desired.
\end{proof}
\section{Concluding Remarks}
We end this note with several remarks. First, the following corollary follows from the corresponding facts about the permutation pattern matching problem, as shown in \cite{BBL}.
\begin{corollary}
The SPPM problem is NP-Complete and the corresponding counting problem is  $\#$P-Complete. \qed
\end{corollary}
It is not clear whether or not pattern matching in set partitions should reduce to pattern matching in permutations. Given two partitions $\sigma$ and $\sigma'$, it seems a much subtler question to produce permutations $p(\sigma)$ and $p(\sigma')$ for which $\sigma$ contains $\sigma'$ if and only if $p(\sigma)$ contains $p(\sigma')$. It would be interesting to further examine this possibility.

There is a weaker notion of containment and avoidance in set partitions which is defined in terms of their corresponding \emph{restricted growth function (RGF)}; see \cite{REU} for definitions. Define the RGF pattern matching problem as the question of determining whether one set partition contains another with respect to this notion of containment.  The next corollary follows from the fact that for the partitions appearing in the proof of Theorem 1, the two notions of containment coincide.
\begin{corollary}
The permutation pattern matching problem is polynomial time reducible to the RGF pattern matching problem. In particular, the RGF pattern matching problem is NP-Complete and the corresponding counting problem is $\#$P-Complete. \qed
\end{corollary}
It would be interesting to refine Theorem 1. Let us discuss an avenue for doing so, following \cite{JK}. For a partition $\pi$, let $Av(\pi)$ denote the set of all partitions which avoid $\pi$. Let $Av(\pi)$-SPPM denote the problem in which one is given an arbitrary set partition $\sigma$ and a set partition $\sigma'$ which is known to avoid $\pi$, and one must decide whether or not $\sigma$ contains $\sigma'$. The complexity of this restricted problem now depends strongly on the pattern $\pi$; for example, we have the following:
\begin{proposition}
There are polynomial time algorithms for solving $Av(\pi)$-SPPM if $\pi = \{\{1,2\}\}$ or $\pi = \{\{1\},\{2\}\}.$
\end{proposition}
\begin{proof}
Let $\sigma$ be an arbitrary set partition, and let $\tau$ be a set partition avoiding $\{\{1,2\}\}$. As $\tau$ avoids $\{\{1,2\}\}$ it cannot have a block of size $\geq 2$, so it must be of the form $\tau = \{\{1\},\{2\},\dots,\{k\}\}$. To determine if $\sigma$ contains a copy of $\tau$ we merely need to verify if $\sigma$ has more than $k$ blocks, and hence $Av(\{\{1,2\}\})$-SPPM can be solved in polynomial time. 

Now suppose $\tau$ avoids $\{\{1\},\{2\}\}$. This requires $\tau$ to consist of a single block, i.e. $\tau = \{\{1,2,\dots,k\}\}$. In particular $\sigma$ avoids $\tau$ if and only if it has no block of size $\geq k$. Thus $Av(\{\{1\},\{2\}\})$-SPPM can also be solved in polynomial time.
\end{proof}
One may find patterns $\pi$ for which $Av(\pi)$-SPPM is NP-Complete by appealing to \cite{JK} and invoking Theorem 1. It would be desirable to give a full classification in this setting, in analogy to Theorem 1.3 of \cite{JK}.
\begin{question}
For which $\pi$ is the $Av(\pi)$-SPPM problem NP-Complete? What if the RGF notion of containment is used?
\end{question}
Finally, we encourage a deeper exploration of hardness and of algorithm design in a wider combinatorial context. For instance, there is a natural notion of containment in ascent sequences; is the corresponding pattern matching problem NP-Complete? What sort of algorithms can one produce for pattern matching in set partitions or other combinatorial objects? Are there natural conjectures which follow from data collected from these algorithms? We believe many questions of this form would provide natural follow up projects.

\textbf{Acknowledgements:} Many thanks to Jason O'Neill for reading a preliminary draft of this paper. The author is grateful to acknowledge funding from the NSF Research Training Group grant DMS-1502651 and from the NSF grant DMS-1849173.
\bibliography{ref}{}

\begin{thebibliography}{AAAH01}

\bibitem[AAAH01]{AAAH}
Michael~H. Albert, Robert E.~L. Aldred, Mike~D. Atkinson, and Derek~A. Holton.
\newblock Algorithms for pattern involvement in permutations.
\newblock In {\em Algorithms and computation ({C}hristchurch, 2001)}, volume
  2223 of {\em Lecture Notes in Comput. Sci.}, pages 355--366. Springer,
  Berlin, 2001.

\bibitem[ALLV16]{ALLV}
Michael Albert, Marie-Louise Lackner, Martin Lackner, and Vincent Vatter.
\newblock The complexity of pattern matching for 321-avoiding and skew-merged
  permutations.
\newblock {\em Discrete Math. Theor. Comput. Sci.}, 18(2):Paper No. 11, 17,
  2016.

\bibitem[AR08]{AR}
Shlomo Ahal and Yuri Rabinovich.
\newblock On complexity of the subpattern problem.
\newblock {\em SIAM J. Discrete Math.}, 22(2):629--649, 2008.

\bibitem[BBL98]{BBL}
Prosenjit Bose, Jonathan~F. Buss, and Anna Lubiw.
\newblock Pattern matching for permutations.
\newblock {\em Inform. Process. Lett.}, 65(5):277--283, 1998.

\bibitem[BL16]{BL}
Marie-Louise Bruner and Martin Lackner.
\newblock A fast algorithm for permutation pattern matching based on
  alternating runs.
\newblock {\em Algorithmica}, 75(1):84--117, 2016.

\bibitem[GM14]{GM}
Sylvain Guillemot and D\'{a}niel Marx.
\newblock Finding small patterns in permutations in linear time.
\newblock In {\em Proceedings of the {T}wenty-{F}ifth {A}nnual {ACM}-{SIAM}
  {S}ymposium on {D}iscrete {A}lgorithms}, pages 82--101. ACM, New York, 2014.

\bibitem[GV09]{GV}
Sylvain Guillemot and St\'{e}phane Vialette.
\newblock Pattern matching for 321-avoiding permutations.
\newblock In {\em Algorithms and computation}, volume 5878 of {\em Lecture
  Notes in Comput. Sci.}, pages 1064--1073. Springer, Berlin, 2009.

\bibitem[Iba97]{Iba}
Louis Ibarra.
\newblock Finding pattern matchings for permutations.
\newblock {\em Inform. Process. Lett.}, 61(6):293--295, 1997.

\bibitem[JK17]{JK}
V\'{\i}t Jel\'{\i}nek and Jan Kyn\v{c}l.
\newblock Hardness of permutation pattern matching.
\newblock In {\em Proceedings of the {T}wenty-{E}ighth {A}nnual {ACM}-{SIAM}
  {S}ymposium on {D}iscrete {A}lgorithms}, pages 378--396. SIAM, Philadelphia,
  PA, 2017.

\bibitem[JM08]{REU}
V\'{\i}t Jel\'{\i}nek and Toufik Mansour.
\newblock On pattern-avoiding partitions.
\newblock {\em Electron. J. Combin.}, 15(1):Research paper 39, 52, 2008.

\bibitem[Kit11]{Kit}
Sergey Kitaev.
\newblock {\em Patterns in permutations and words}.
\newblock Monographs in Theoretical Computer Science. An EATCS Series.
  Springer, Heidelberg, 2011.
\newblock With a foreword by Jeffrey B. Remmel.

\bibitem[LF17]{LF}
Zhicong Lin and Shishuo Fu.
\newblock On 1212-avoiding restricted growth functions.
\newblock {\em Electron. J. Combin.}, 24(1):Paper No. 1.53, 20, 2017.

\bibitem[Sag10]{Sag}
Bruce~E. Sagan.
\newblock Pattern avoidance in set partitions.
\newblock {\em Ars Combin.}, 94:79--96, 2010.

\bibitem[Ten]{DPPA}
Bridget~E. Tenner.
\newblock Database of permutation pattern avoidance.
\newblock Published electronically at {\tt
  https://math.depaul.edu/\textasciitilde bridget/patterns.html}.
\newblock Accessed: 2020-04-02.

\bibitem[YS05]{YS}
V.~Yugandhar and Sanjeev Saxena.
\newblock Parallel algorithms for separable permutations.
\newblock {\em Discrete Appl. Math.}, 146(3):343--364, 2005.

\end{thebibliography}
\bibliographystyle{alpha}
\end{document}